\documentclass[12pt,reqno]{amsart}

\usepackage{amsmath,amssymb,amsfonts,amsthm,amscd,amstext,amsxtra,amsopn,array,url,verbatim,mathrsfs}
\usepackage{graphicx}
\usepackage[mathscr]{euscript}
\usepackage{amsmath,amssymb,amsfonts,amsthm,amssymb,amscd,url,amstext,amsxtra,amsopn}
\usepackage{verbatim}
\usepackage{fullpage}
\usepackage{times}

\usepackage[colorlinks=true,linkcolor=black,citecolor=black]{hyperref}

\newtheorem{theorem}{Theorem}[section]
\newtheorem{lemma}[theorem]{Lemma}

\newtheorem{prop}[theorem]{Proposition}
\newtheorem{cor}[theorem]{Corollary}
\newtheorem{conj}[theorem]{Conjecture}

\theoremstyle{definition}

\newtheorem{example}[theorem]{Example}

\newtheorem{notation}[theorem]{Notation}

\newtheorem{remarks}[theorem]{Remarks}

\numberwithin{equation}{section}

\renewcommand{\mod}[1]{{\ifmmode\text{\rm\ (mod~$#1$)}\else\discretionary{}{}{\hbox{ }}\rm(mod~$#1$)\fi}}

\newcommand{\lcm}{{\rm l.c.m.}}
\newcommand{\degg}{{\rm deg}}

\begin{document}

\title{On the greatest prime factor of some divisibility sequences}


\author{Amir Akbary }
\address{Department of Mathematics and Computer Science \\
        University of Lethbridge \\
        Lethbridge, AB T1K 3M4 \\
        Canada}
\curraddr{}
\email{amir.akbary@uleth.ca}


\thanks{Research of the authors is partially supported by NSERC}

\author{Soroosh Yazdani}
\curraddr{}
\email{syazdani@gmail.com}


\date{}

\begin{abstract}
Let $P(m)$ denote the greatest prime factor of $m$. For integer $a>1$, M. Ram Murty and S. Wong proved that, 
under the assumption of the ABC conjecture,
$$P(a^n-1)\gg_{\epsilon, a} n^{2-\epsilon}$$
for any $\epsilon>0$. We study analogues results for the corresponding divisibility sequence over the function field $\mathbb{F}_q(t)$ and for some divisibility sequences associated to elliptic curves over the rational field $\mathbb{Q}$. 
\vskip .25cm
\hfill \textit{In honor of M. Ram Murty on his sixtieth birthday}

\end{abstract}

\maketitle

\section{introduction and results}

Let $P(m)$ denote the greatest prime factor of the integer $m$. Several authors investigated the size of $P(2^n-1)$. In \cite[Lemma 3]{S} Schinzel proved that $$P(2^n -1) \geq 2n+1,~~{\rm for}~~n\geq 13.$$ In 1965, Erd\H{o}s \cite[p. 218]{Erdos} conjectured that $$\lim_{n\rightarrow \infty} \frac{P(2^n-1)}{n}=\infty.$$ This conjecture has been recently resolved by Stewart \cite{St}. More generally, for integers $a>b>0$, one can consider lower bounds in terms of $n$ for $P(a^n-b^n)$. The first general result on this problem is due to Zsigmondy \cite{Z} and independently to Birkhoff and Vandiver \cite{BV} who showed that
$$P(a^n-b^n) \geq n+1.$$
The best known result on this problem is the recent result of Stewart \cite[Formula (1.8)]{St} that states
$$ P(a^n-b^n) \geq n^{1+\frac{1}{104\log\log{n}}},$$
for $n$ sufficiently large in terms of the number of distinct prime factors of $ab$. Note that the above lower bound for $a=2$ and $b=1$ implies Erd\H{o}s' conjecture. We expect that $P(a^n-b^n)$ be much larger than $n^{1+\epsilon(n)}$, where $\epsilon(n)\rightarrow 0$ as $n\rightarrow \infty$. Here we describe a heuristic argument in support of this claim. To simplify our notation, we now focus on $a^n-1$. Similar observations hold for the sequence $a^n-b^n$.

{\tiny

\begin{table}[h]    
\centering
\begin{tabular}{ ||l|c||l|c||l|c| }
\hline
 $n$ & Factorization of $2^n-1$ & $n$  & Factorization of $2^n-1$ & $n$&  Factorization of $2^n-1$\\ 
\hline
$1$ &$1$ & $11$ &$23 \times 89$ & $21$ & $7^2\times 127 \times 337$\\
$2$ &$3$& $12$ &$3^2 \times 5 \times 17$ & $22$ & $3\times 23\times 89\times 683$\\
$3$ &$7$& $13$ & $8191$ & $23$ & $47 \times 178481$\\
$4$ &$3\times 5$ & $14$ &$3\times 43\times 127$ & $24$ &$ 3^2\times 5\times 7\times 13 \times 17 \times 241$\\
$5$ &$31$ & $15$ &$7\times 31\times 151$ & $25$ &$31\times 601 \times 1801$\\
$6$ &$3^2 \times 7$ & $16$ &$3\times 5\times 17\times 257$ & $26$ & $3 \times 2731 \times 8191$\\
$7$ &$127$& $17$ &$131071$ & $27$ & $7\times 73 \times 262657$\\
$8$ &$3\times 5\times 17$ & $18$ &$3^3 \times 7 \times  19 \times 73$ & $28$ & $ 3\times 5 \times 29 \times 43 \times 113 \times  127$\\
$9$ &$7\times 73 $ & $19$ &$524287$ & $29$ & $ 233 \times 1103 \times 2089$\\
$10$ &$3\times 11 \times 31$ & $20$ &$3\times 5^2 \times 11 \times 31 \times 41$ & $30$&    $3^2 \times 7 \times 11 \times 31 \times 151 \times 331 $\\
\hline
\end{tabular}
\end{table}

}

We write $a^n-1=u_n v_n$, where $u_n$ is power-free (square-free) and $v_n$ is power-full (the exponent of prime divisors of $v_n$ in the prime factorization of $v_n$ are greater than 1). Now if we denote the number of prime divisors of an integer $m$ by $\omega(m)$, then we can find a lower bound for $P(a^n-1)$ in terms of $u_n$ and $\omega(a^n-1)$ as follows.
We have $$P(a^n-1)^{\omega(a^n-1)}\geq P(u_n)^{\omega(a^n-1)}\geq u_n,$$
or equivalently 
\begin{equation}
\label{lower}
\log{P(a^n-1)} \geq \frac{\log{u_n}}{\omega(a^n-1)}.
\end{equation}
Thus a lower bound for $u_n$ and an upper bound for $\omega(a^n-1)$ furnishes a lower bound for the greatest prime factor of $a^n-1$.

By looking at the factorization of $a^n-1$ for different values of $a$ and $n$ (see the above table of prime factorization of $2^n-1$ for $1\leq n \leq 30$), we speculate the following two statements regarding the factorization of $a^n-1$. 

\noindent{\sc First Observation:} {\it The power-full part of $a^n-1$ is small.}

\noindent {\sc Second Observation:} {\it The number of prime factors of $a^n-1$ is small.}

\noindent These together with \eqref{lower} imply that $P(a^n-1)$ is large. 

The above argument can be quantified by using well-known   conjectures. Here we recall the celebrated $ABC$ conjecture and a conjecture of Erd\H{o}s on ${\rm ord}_p(a)$, the multiplicative order of an integer $a$ modulo a prime $p$.

\begin{conj}[{ABC conjecture of Masser-Oesterl\'{e}}]
\label{ABC-conj}
Let $A, B, C \in \mathbb{Z}$ be relatively prime integers satisfying $A+B+C=0$. Then for every $\epsilon>0$,
$$\max\{ |A|, |B|, |C| \} \ll_\epsilon \left(\prod_{\pi\mid ABC} \pi  \right)^{1+\epsilon}.$$
\end{conj}

\begin{conj}[{Erd\H{o}s}]
\label{Erdos}
For an integer $a$ and a  positive integer $r$, let $$E_a(r)=\#\{p~{\rm prime};~{\rm ord}_p(a)=r\}.$$ 
Then for every $\epsilon>0$ we have $$E_a(r) \ll_\epsilon r^\epsilon.$$
\end{conj}
Conjecture \ref{ABC-conj} is stated in \cite{Masser}. Conjecture \ref{Erdos} is formulated in \cite{Econjecture} for $a=2$.

In \cite[Lemma 7]{Silverman} Silverman provided the following statement in support of our first observation. 

\begin{prop}[{Silverman}]
Let $a^n-1=u_n v_n$ be the decomposition of $a^n-1$ as the product of the power-free part $u_n$ and power-full part $v_n$. Then for any $\epsilon>0$, under the assumption of Conjecture \ref{ABC-conj}, we have $$v_n \ll_{\epsilon, a} a^{\epsilon n}.$$
\end{prop}
From this proposition we conclude that under the assumption of Conjecture \ref{ABC-conj}, the power-free part of $a^n-1$ is large. More precisely for $\epsilon>0$ we have $$u_n\gg_{\epsilon, a} a^{(1-\epsilon)n}.$$

We know that the normal order of $\omega(n)$ is $\log\log{n}$. From here we may speculate that $\omega(a^n-1) \approx \log{n}.$ However as a consequence of a theorem of Prachar we can show that $\omega(a^n-1)$ is greater than $\log{n}$ for infinitely many values of $n$. More precisely, in \cite[Satz 2]{Pr}, Prachar proves that 
$$\#\{p~{\rm prime};~(p-1)\mid n\}\geq \exp\left(\frac{c\log{n}}{(\log\log{n})^2} \right),$$
for some 
$c>0$ and for infinitely many $n$.  This implies that there exists $c>0$ such that $$\omega(a^n-1)\geq \exp\left(\frac{c\log{n}}{(\log\log{n})^2} \right),$$ 
for infinitely many $n$.
In \cite{FM}, Felix and Murty observed that $$\omega(a^n-1)=\#\{p~{\rm prime};~ p\mid a^n-1 \}=\sum_{d\mid n} E_a(d).$$
So under the assumption of Conjecture \ref{Erdos}, we have
$$\omega(a^n-1)\ll_\epsilon n^\epsilon.$$

The above observations are summarized in the following theorem (see \cite[Section 5]{FM}).

\begin{theorem}[{Felix-Murty}]
Under the assumptions of Conjecture \ref{ABC-conj} and Conjecture \ref{Erdos}, for any $\epsilon>0$, we have $$P(a^n-1)\gg _{\epsilon,a} a^{n^{1-\epsilon}}.$$
\end{theorem}

It is interesting to note that the small size of $\omega(a^n-1)$ plays a crucial role in the proof of the above theorem. In fact under the assumption of Conjecture \ref{ABC-conj} and by employing the unconditional upper bound $\omega(a^n-1)\ll n/\log{n}$ and \eqref{lower}, we get $$P(a^n-1)\gg_{\epsilon,a} n^{1-\epsilon},$$
which is weaker than  known unconditional bounds. So it was remarkable that in 2002, Murty and Wong \cite[Theorem 1]{MW}, without appealing to any bound for $\omega(a^n-1)$, could prove the following theorem.

\begin{theorem}[{Murty-Wong}]
\label{MW}
Under the assumption of Conjecture \ref{ABC-conj}, for any $\epsilon>0$, we have $$P(a^n-1)\gg _{\epsilon,a} {n^{2-\epsilon}}.$$
\end{theorem}


The sequence $a^n-1$ is an example of a divisibility sequence. 
A sequence $(d_n)$ of integers is called a \emph{divisibility sequence} if 
$$m\mid n \Rightarrow d_m \mid d_n.$$
In this paper, under certain conditions, we extend Murty-Wong's theorem to divisibility sequences other than $a^n-1$.

Our first generalization is a function field analogue of Theorem \ref{MW}. Let $\mathbb{F}_q$ be a finite field of characteristic $p$. For $b(t)\in \mathbb{F}_q[t]$, let $G(b(t))$ be the greatest of the degrees of the irreducible factors of $b(t)$. 
Then we ask how large can $$ G(a(t)^n-1)$$ be? Here we prove the following result related to this question.

\begin{theorem}
\label{MW-functionfield} Let $a(t)\in \mathbb{F}_q[t]$ be a polynomial that is not a perfect $p$-th power.
Let $$\mathcal{S}_{\alpha}=\{{\rm prime}~\ell;~\ell\neq p,~{\rm ord}_\ell(q)\geq \ell^\alpha \}.$$
Then for $\epsilon>0$ we have the following assertions.

\noindent (i) There is a constant $C=C(\epsilon, q, a(t))$ such that
$$G(a(t)^\ell-1)\geq(1+\alpha-\epsilon)\log_q{\ell}+C,~\textrm{for all}~\ell \in \mathcal{S}_{\alpha}.$$

\noindent (ii) There is a constant $C=C(\epsilon, q, a(t))$ such that for all primes $\ell\leq x$, except possibly $o(x/\log{x})$ of them,  we have $$ G(a(t)^{\ell}-1)\geq (3/2-\epsilon) \log_q{\ell}+C.$$

\noindent (iii) Assume that for all integers $d\geq 1$ the generalized Riemann hypothesis (GRH) holds for the Dedekind zeta function of $\mathbb{Q}(\zeta_d, q^{1/d})$, where $\zeta_d$ is a primitive $d$-th root of unity. Then there is a constant $C=C(\epsilon, q, a(t))$ such that for all primes $\ell\leq x$, except possibly $o(x/\log{x})$ of them,  we have $$ G(a(t)^{\ell}-1)\geq (2-\epsilon) \log_q{\ell}+C.$$
\end{theorem}

\begin{remarks}
(i) Following the proof of Part (i) of the above theorem we can show that an assertion similar to Part (i) holds for ${G(a(t)^n-1)}$, as long as integer $n$ belongs to 
$$\{n;~p\nmid n,~{\rm ord}_m(q)\geq m^\alpha~{\rm for~all~}m\mid n~{\rm and}~m>n^{1-\epsilon} \},$$
where $\epsilon>0$ is a fixed constant. 

\noindent (ii) Unlike Theorem \ref{MW}, Parts (i) and (ii) of the above theorem are unconditional. This is due to a known version of the ABC conjecture, due to Mason, for the function fields (see Theorem \ref{mason}). The condition that $a(t)$ is not a perfect $p$-th power is needed for application of Mason's theorem.

\noindent (iii) The above theorem establishes an intimate connection between the growth of  degree of \\${G(a(t)^n-1)}$ in a function field $\mathbb{F}_q[t]$ and the multiplicative order of integer $q$ modulo $n$. This is a common feature in many function field problems that their study ties together with the study of problems in integers. A notable example is the appearance of Romanoff's theorem  in Bilharz's proof of Artin's primitive root conjecture over function fields (see \cite[Chapter 10]{R}).   

\noindent (iv) The function ${\rm ord}_n(q)$ has an erratic behavior, and although most of the times it is large it can take small values too. For example if we assume there are infinitely many Mersenne primes then there are infinitely many primes $\ell$ for which ${\rm ord}_\ell(2)$ is as small as $\log{\ell}$.

\noindent (v) Part (iii) of the above theorem is comparable with Murty-Wong's theorem (Theorem \ref{MW}). However the statement is weaker in the sense that integers are replaced by almost all prime numbers. Also Part (iii) is conditional upon the GRH while Murty-Wong's is conditional upon the ABC conjecture. It is debatable which one of these conjectures is harder than the other.  

\noindent (vi) Note that $G(a(t)^n-1)$  when $n$ is a multiple of $p$ behaves differently, as $G(a(t)^{mp}-1)=G(a(t)^{m}-1).$

\end{remarks}





Our next example of a divisibility sequence is related to elliptic curves.
Let $E$ be an elliptic curve given by the  Weierstrass equation $$y^2 =x^3+A x+B,$$
where $A, B\in \mathbb{Z}$.
Let $E(\mathbb{Q})$ be the group of rational points of $E$.
It is known that any rational point on $E$ has an expression in the form $(a/d^2, b/d^3)$ with $(a,d)=(b,d)=1$ and $d\geq 1$ (see \cite[p. 68]{ST}).   
Let $Q$ be a rational point of infinite order in $E(\mathbb{Q})$. 
The \emph{elliptic denominator sequence} $(d_n)$ associated to $E$ and $Q$ is defined by
$$nQ=\left(\frac{a_n}{d_n^2}, \frac{b_n}{d_n^3}  \right).$$
One can show that $(d_n)$ is a divisibility sequence. 
\begin{example}
\label{Jeff}
Let $E$ be given by $y^2=x^3-11$. Then $Q=(3, 4)$ is a point of infinite order in $E(\mathbb{Q})$. Let $(d_n)$ be the denominator sequence associated to $E$ and $Q$. We record the prime power factorization of $d_n$ for $1\leq n \leq 17$.

{\tiny
\begin{table}[h]    
\centering
\begin{tabular}{ ||l|l| }
\hline
 $n$ & {\rm Factorization of} $d_n$ {\rm associated to} $y^2=x^3-11$ {\rm and} $(3,4)$ \\ 
\hline
$1$ & $1$ \\ 
$2$ & $2^{3}$ \\
$3$ & $3^{2} \cdot 17$ \\ 
$4$ & $2^{4} \cdot 37 \cdot 167$ \\
$5$ & $449 \cdot 104759$ \\ 
$6$ & $2^{3} \cdot 3^{2} \cdot 5 \cdot 17 \cdot 23 \cdot 1737017$ \\
$7$ & $19 \cdot 433 \cdot 2689 \cdot 8819 \cdot 40487$ \\ 
$8$ & $2^{5} \cdot 37 \cdot 167 \cdot 245519 \cdot 3048674017$\\ 
$9$ & $3^{3} \cdot 17 \cdot 861139 \cdot 638022143238323743$\\ 
$10$ & $2^{3} \cdot 29 \cdot 449 \cdot 39631 \cdot 54751 \cdot 104759 \cdot 117839 \cdot 181959391$\\
$11$ & $11 \cdot 331 \cdot 2837 \cdot 4423 \cdot 4621 \cdot 687061 \cdot 40554559 \cdot 105914658299$\\
$12$ & $2^{4} \cdot 3^{2} \cdot 5 \cdot 17 \cdot 23 \cdot 37 \cdot 107 \cdot 167 \cdot 1288981 \cdot 1737017 \cdot 64132297 \cdot 7428319481306593$\\
$13$ & $7 \cdot 31 \cdot 233 \cdot 452017 \cdot 104847601 \cdot 26215872615271 \cdot 403453481668667999145407$\\
$14$ & $2^{3} \cdot 19 \cdot 41 \cdot 211 \cdot 433 \cdot 503 \cdot 2309 \cdot 2689 \cdot 4451 \cdot 8819 \cdot 28813 \cdot 40487 \cdot 42859 \cdot 306809 \cdot 404713 \cdot 909301 \cdot 35196247$\\
$15$ & $3^{2} \cdot 17 \cdot 449 \cdot 631 \cdot 29819 \cdot 104759 \cdot 258659 \cdot 1331521 \cdot 2681990178080401065344970115363369337376832169$\\
$16$ & $2^{6} \cdot 37 \cdot 167 \cdot 431 \cdot 3169 \cdot 49537 \cdot 245519 \cdot 3048674017 \cdot 606437794508831 \cdot 3321240163385870449 \cdot 21659973345967709759$\\
$17$ & $101^{2} \cdot 606899 \cdot 1865887 \cdot 141839057 \cdot 383168404657137063963767 \cdot 199169555888386471211683643669332910982224853163$\\
\hline
\end{tabular}
\end{table}
}

A glance at the above table shows that assertions similar to Observations 1 and 2 for $a^n-1$ may hold for $d_n$. In fact, following an argument similar to the case $a^n-1$, one may speculate that, for any $\epsilon>0$, we have
$$\log{P(d_n)} \gg_{\epsilon, E, Q} n^{2-\epsilon},$$
where the implied constant depends on $E$, $Q$, and $\epsilon$.
\end{example}



We will prove the following conditional lower bound for $P(d_n)$ for certain elliptic curves. 

\begin{theorem}
\label{elliptic-theorem}
Let $E$ be an elliptic curve over $\mathbb{Q}$ of $j$-invariant $0$ or $1728$. For a point of infinite order $Q\in E(\mathbb{Q})$, let $(d_n)$ be the elliptic denominator sequence associated to $E$ and $Q$. Assume Conjecture \ref{ABC-conj}. Then for any $\epsilon>0$ we have
$$P(d_n)\gg_{\epsilon, E, Q} n^{3-\epsilon},$$ 
or equivalently $$\log{P(d_n)}\geq {(3-\epsilon)} \log{n}+O_{\epsilon,E, Q}(1).$$ 
\end{theorem}

Some authors call the above sequence $(d_n)$ an {elliptic divisibility sequence}. We decided to call them elliptic denominator sequences to differentiate them from the classical elliptic divisibility sequences defined and studied by Ward \cite{W}.  A divisibility sequence $(w_n)$ is called an \emph{elliptic divisibility sequence} if $w_1=1$ and, for $n>m$, $(w_n)$ satisfies the recurrence
$$w_{n+m}w_{n-m}=w_{n+1}w_{n-1}w_m^2-w_{m+1}w_{m-1}w_n^2.$$
The discriminant ${\rm Disc}(w)$ of an elliptic divisibility sequence $(w_n)$ is defined by the formula
\begin{multline*}
{\rm Disc}(w)=w_4 w_2^{15}-w_3^3w_2^{12}+3w_4^2 w_2^{10}-20 w_4 w_3^3 w_2^7
+3w_4^3 w_2^5+16w_3^6 w_2^4+8w_4^2 w_3^3w_2^2+w_4^4.
\end{multline*}
An elliptic divisibility sequence is called \emph{non-singular} if $w_2w_3{\rm Disc}(w)\neq 0$. There is a close connection between non-singular elliptic divisibility sequences and elliptic curves. More precisely a theorem of Ward states that for any non-singular elliptic divisibility sequence $(w_n)$, there exist an elliptic curve $E$ and a point $Q\in E(\mathbb{Q})$ such that $(w_n)$ can be realized as the values of certain elliptic functions on $E$ evaluated at $Q$ (see \cite[Theorems 12.1 and 19.1]{W}). Moreover $E$ and $Q$ can be explicitly constructed in terms of $w_2$, $w_3$, and $w_4$ (see \cite[Appendix A]{SS}). We call the pair $(E,Q)$ given in
\cite[Appendix A]{SS}, the \emph{curve point} associated to $(w_n)$. In addition if $(d_n)$ is the denominator sequence associated to $E$ and $Q$, we can show that $d_n\mid w_n$.
As an immediate consequence of Theorem \ref{elliptic-theorem} and relation $d_n\mid w_n$ we have the following result.

\begin{cor}
\label{elliptic-corollary}
Let $(w_n)$ be a non-singular elliptic divisibility sequence with the associated curve point $(E,Q)$.  Suppose that $E$ has $j$-invariant $0$ or $1728$ and $Q$ has infinite order. Then under the assumption of Conjecture \ref{ABC-conj} we have
$$P(w_n)\gg_{\epsilon, E, Q} n^{3-\epsilon},$$ 
for any $\epsilon>0$, or equivalently $$\log{P(w_n)}\geq {(3-\epsilon)} \log{n}+O_{\epsilon,E, Q}(1).$$ 
\end{cor}










In Section \ref{section2} we prove Theorem \ref{MW-functionfield}.  The proof of Theorem \ref{elliptic-theorem} is given in Section \ref{section3}.

\begin{notation}
Throughout the paper $p$ and $\ell$ denote primes, $q=p^r$, $\mathbb{F}_q$ is the finite field of $q$ elements, $\mathbb{F}_q[t]$ and $\mathbb{F}_q(t)$ are the ring of polynomials and the function field with coefficients in $\mathbb{F}_q$. We let ${\rm ord}_p (a)$ be the multiplicative order of integer $a$ modulo $p$. 
The letter $\pi$ denotes either a rational prime or a monic irreducible polynomial in $\mathbb{F}_q$ (for simplicity we write $\pi(t)$ as $\pi$ in this case). For a polynomial $b(t)\in \mathbb{F}_q[t]$, we let $G(b(t))$ 
be the greatest of the degrees of the irreducible factors of $b(t)$. For a monic irreducible polynomial $\pi$ and a polynomial $a(t)\in \mathbb{F}_q[t]$ we let ${\rm o}_\pi(a)$ be the multiplicative order of $a(t)$ modulo $\pi$.
For an elliptic curve $E$ defined over $\mathbb{Q}$ and a good prime $\pi$, we denote the number of points of reduction modulo $\pi$ of $E$ by $n_\pi(E)$. We denote the group of $\mathbb{Q}$-rational points of $E$ by $E(\mathbb{Q})$ and the discriminant of $E$ by $\Delta_E$. 
We let ${\rm o}_\pi(Q)$ be the order of the point $Q\in E(\mathbb{Q})$ modulo $\pi$. We denote the elliptic denominator sequence associated to $E$ and $Q$ by $(d_n)$ and we let $D_n$ be  the primitive divisor of $d_n$.
The functions $\tau(n)$, $\omega(n)$, and $P(n)$ are the divisor function, the number of distinct prime divisors function, and the greatest prime factor function.  
For two real functions $f(x)$ and $g(x)\neq 0$, we use the notation $f(x)=O_s(g(x))$, or alternatively $f(x)\ll_s g(x)$, if $|f(x)/g(x)|$ is  bounded by a constant, depending on a parameter $s$, as $x\rightarrow \infty$. Finally we write $f(x)=o(g(x))$ if $\lim_{x\rightarrow \infty} f(x)/g(x)=0$. 
\end{notation}


\section{Polynomial Case}
\label{section2}
In this section we assume that $a$, $b$, $c$ are polynomials in $\mathbb{F}_q[t]$, where $\mathbb{F}_q$ is the finite field of $q=p^r$ elements. (For simplicity from now on we drop the variable $t$ in our notation for polynomials.) We denote a monic irreducible polynomial by $\pi$, and we call such a polynomial a \emph{prime polynomial}. We need the next three lemmas in the proof of Theorem \ref{MW-functionfield}.

The following assertion, which is analogous to the $ABC$ Conjecture, holds in $\mathbb{F}_q(t)$.

\begin{lemma}[Mason]
\label{mason}
Let $a$, $b$, $c\in \mathbb{F}_q[t]$ be  relatively prime polynomials that are not all perfect $p$-th power. If $$c=a+b,$$
then
$$\max\{{\rm deg}(a), {\rm deg}(b), {\rm deg}(c)\} \leq {\rm deg}\left(\prod_{\pi \mid abc} \pi\right)-1.$$
\end{lemma}
\begin{proof}
See \cite[p. 156]{M}.
\end{proof}

We let $C_1=a-1$ and  $C_\ell=(a^\ell -1)/{(a-1)}$, where $\ell$ is a prime integer. For a prime polynomial $\pi$, we denote the multiplicative order of $a$ mod $\pi$ by ${\rm o}_\pi(a)$. Then for a  $\pi$, where $\pi \nmid a$, and a prime $\ell\neq p$, we have 
\begin{equation}
\label{order-relation}
{\rm o}_\pi(a)=\ell \iff \pi \mid C_\ell.
\end{equation}

The first part of the next lemma can be considered as an analogue of the prime number theorem in function fields. 
 
\begin{lemma}
\label{PNT}
Let $\pi$ denote a prime in $\mathbb{F}_q(t)$. 

\noindent (i) For positive integer $k$, we have
$$\#\{\pi;~{\rm deg}(\pi)=k\}=\frac{q^k}{k}+O\left(\frac{q^{\frac{k}{2}}}{k}  \right).$$
(ii) For positive integers $k$ and $N$ we have 
$$\#\{\pi;~{\rm deg}(\pi)\leq N~{\rm and}~k\mid \degg(\pi) \} \ll \frac{q^N}{k}.$$
\end{lemma}
\begin{proof}
(i) See \cite[Theorem 2.2]{R}.

\noindent (ii) From Part (i) we have
$$\sum_{\substack{{\pi;~ k\mid \degg(\pi)}\\{\degg(\pi)\leq N}}} 1 \ll  \frac{1}{k} \left( q^k+q^{2k}+\cdots+q^{[N/k]k} \right)\ll  \frac{ q^N}{k}.$$
\end{proof}

Recall that for an integer $m$ we denote the multiplicative order of $m$ modulo  $\ell$ by ${\rm ord}_\ell (m)$. The next lemma provides information on the size of ${\rm ord}_\ell (m)$. 

\begin{lemma}[{Erd\H{o}s-Murty}]
\label{order-i}
Let $m\in \mathbb{Z}\setminus \{0,\pm 1\}$.  Then we have the following statements.

(i)  Let $\epsilon: \mathbb{R}^+\rightarrow \mathbb{R}^+$  be a function such that $\epsilon(x) \rightarrow 0$ as $x \rightarrow \infty$. 
Then
$${\rm ord}_\ell(m)\geq \ell^{1/2+\epsilon(\ell)}$$
for all but $o(x/\log{x})$ primes $\ell\leq x$.

(ii) Let $f:\mathbb{R}^+\rightarrow \mathbb{R}^+$ be a function such that $f(x) \rightarrow \infty$ as $x\rightarrow \infty.$ For each integer $d\geq 1$ we assume that the generalized Riemann hypothesis (GRH) holds for the Dedekind zeta function of $\mathbb{Q}(\zeta_d, m^{1/d}),$ where $\zeta_d$ denote a primitive $d$-th root of unity. Then for all but $o(x/\log{x})$ primes $\ell \leq x$, we have $${\rm ord}_\ell(m)\geq \frac{\ell}{f(\ell)}.$$
\end{lemma}

\begin{proof}
These are Theorems 1 and 4 in \cite{EM}.
\end{proof}

We are ready to prove Theorem \ref{MW-functionfield}.

\begin{proof}[Proof of Theorem \ref{MW-functionfield}]
Let $\ell$ be a prime different from $p$. We start by setting $C_m=U_m V_m$, for $m=1$ and $\ell$, where $U_m$ is the power-free part of $C_m$ and $V_m$ is the power-full part of $C_m$. (Recall that $C_1=a-1$ and $C_\ell=(a^\ell-1)/(a-1)$.) Observe that $(a^\ell-1)+1=a^\ell$. Thus, since $a$ is not a perfect $p$-th power, by Lemma \ref{mason} we have
\begin{eqnarray*}
\max\{\deg{(a^\ell-1)}, \deg{((a^\ell))}\}&\leq& \deg \left(\prod_{\pi \mid a C_1 C_\ell} \pi \right)-1\\
&\leq& \sum_{\pi | aC_1 C_\ell} \deg(\pi) -1\\
&\leq& \deg(a)+\sum_{m\mid \ell}\left(\deg(U_m)+{\frac{\deg(V_m)}{2}}
\right)-1
\end{eqnarray*}
Employing $\max\{\deg{(a^\ell-1)}, \deg{(a^\ell)}\} = \deg{(a^\ell-1)}$ and $\sum_{m\mid \ell}\deg(U_m V_m)=\deg(a^\ell-1)$ in the above inequality yields $$\sum_{m\mid\ell}\deg{(V_m)} \leq 2 \deg{(a)}-2.$$
Thus we have $\sum_{m\mid \ell}\deg(U_m) \gg_a \ell$, where the implied constant depends on $a$.

From here, we have
\begin{eqnarray}
\label{ff-main-1}
\nonumber
\ell &\ll_{a}&1+ \deg{(U_\ell)} \\
&\ll_{a}& 1+\sum_{\substack{{\pi\mid U_\ell}\\{\deg(\pi)\leq G(a^\ell -1)}}} \deg{(\pi)}. 
\end{eqnarray}

From \eqref{order-relation} we know that for prime $\ell \neq p$ if $\pi\mid U_\ell$ and $\pi\nmid a$ then $\ell= {\rm o}_\pi(a)$. On the other hand we know that $ {\rm o}_\pi(a) \mid \#\left(\mathbb{F}_q[t]/\langle \pi \rangle \right)^{\times}=q^{\degg(\pi)}-1.$ 
Therefore for such prime divisor $\pi$ of $U_\ell$ we have
$$q^{\degg(\pi)} \equiv 1~({\rm mod}~\ell) \Rightarrow {\rm ord}_\ell(q) \mid \deg{(\pi)}.$$ 
Now from Part (ii) of Lemma \ref{PNT} we conclude that
$$\sum_{\substack{
{\pi\mid U_\ell,~\pi\nmid a}\\{\deg{(\pi)}\leq G(a^\ell -1)}
}}
1\ll \frac{q^{G(a^\ell -1)}}{{\rm ord}_\ell(q)}.$$
Applying the latter inequality in \eqref{ff-main-1}, under the assumption of ${\rm ord}_\ell(q)\geq \ell^{\alpha}$ yields  
\begin{equation*}
\label{main2}
\ell \ll_{a}  1+\frac{G(a^\ell -1)}{\ell^\alpha} q^{G(a^\ell -1)}.
\end{equation*}
From here (i) follows. For (ii) it is enough to observe that, by Part (i) of Lemma \ref{order-i}, the set of primes $\ell$ with ${\rm ord}_\ell(q)\geq \ell^{1/2}$ has density one. For (iii) we note that, by part (ii) of Lemma \ref{order-i}, under the assumption of GRH, the set of primes $\ell$ with ${\rm ord}_\ell(q)\geq \ell/\log{\ell}$ has density one.
\end{proof}

\section{Elliptic Curve Case}
\label{section3}
We review some properties of elliptic denominator sequences associated to elliptic curves and rational points on them. Let
$E$ be an elliptic curve defined over $\mathbb{Q}$.  We assume that $E$ is given by a Weierstrass equation whose coefficients are integers. We denote the discriminant of $E$ by $\Delta_E$. 
Let $Q$ be a point of infinite order in $E(\mathbb{Q})$, and $\mathscr{O}$ denote the point at infinity.
For a prime  $\pi\nmid \Delta_E$  in $\mathbb{Q}$,  let ${\rm o}_\pi(Q)$ denote the order of the point $Q$ modulo  $\pi$. In other words ${\rm o}_\pi(Q)$ is the smallest integer $m\geq 1$ such that $mQ\equiv \mathscr{O}$ (mod $\pi$). Let $n_\pi(E)$ be the number of points of reduction modulo $\pi$ of $E$ over the finite field $\mathbb{F}_\pi$. 

Recall that the elliptic denominator sequence $(d_n)$ associated to $E$ and a non-torsion point $Q\in E(\mathbb{Q})$ is defined by
$$nQ=(a_n /d_n^2, b_n/d_n^3).$$

Let $D_n$ be the largest divisor of $d_n$ which is relatively prime to $d_1 d_2 \cdots d_{n-1}$.  $D_n$ is called the \emph{primitive divisor} of $d_n$. It is clear that for a prime $\pi$ of good reduction 
\begin{equation}
\label{order}
{\rm o}_\pi(Q)=n \iff \pi\mid D_n.
\end{equation}
Let $D_n=U_n V_n$ be the decomposition of the {primitive divisor} $D_n$ of $d_n$ to the power-free part $U_n$ and power-full part $V_n$. 
The following lemma summarizes some basic properties of the sequence $(d_n)$.
\begin{lemma}
\label{d_n and D_n}
\begin{enumerate}
\item With the above notation, we have 
$$\prod_{\substack{{\pi\mid d_n}\\{\pi\nmid \Delta_E}}}\pi=\prod_{\substack{{\pi\mid \left(\prod_{m\mid n} D_m\right)}\\{\pi \nmid \Delta_E}}} \pi.$$
\item $\left(\prod _{m | n} D_m \right)\mid d_n$.
\item If $\pi \nmid  \Delta_E$ and $\pi\mid U_n$, then ${\rm o}_\pi(Q)=n$.
\end{enumerate}
\end{lemma}
\begin{proof}
(a) We want to prove that for any integer $n$, a prime 
$\pi$ divides $d_n$ if and only if $\pi \mid D_m$ for some $m \mid n$. Let $m\mid n$ and assume that $\pi \mid D_m$. Then $\pi \mid d_m$ which implies $\pi \mid d_n$ since $(d_n)$ is a divisibility sequence. Conversely, assume that $\pi\nmid \Delta_E$ and $\pi \mid d_n$. Then $nQ \equiv \mathscr{O} \pmod \pi$. Hence ${\rm o}_\pi(Q) \mid n$, which implies that $\pi \mid D_{{\rm o}_\pi(Q)}$. This proves the desired result. 

(b) Note that for every, $m \mid n$, we have that $D_m \mid d_m \mid d_n$. Therefore $\lcm(D_m)_{m\mid n} \mid d_n$. However, by construction, $D_m$'s are relatively prime to each other. Therefore $\lcm (D_m)_{m\mid n} = \prod_{m|n} D_m$.

(c) This is clear from \eqref{order}, since $\pi \mid U_n \mid D_n$. 
\end{proof}

We need two more lemmas before proving our result for elliptic curves.

\begin{lemma}
\label{size}
We have
\begin{equation*}
n^2 \ll_{E, Q} \log{d_n} \ll_{E, Q} n^2,
\end{equation*}
and moreover $$\log{D_n} \gg_{E, Q} n^2.$$
\end{lemma}
\begin{proof} From \cite[Lemma 8]{Silverman} we know that for any $\epsilon>0$,  
\begin{equation*}
\label{d_n upper bound}
(1-\epsilon)n^2 \hat{h}(Q)+O_{\epsilon, E}(1)\leq\log{d_n} \leq n^2 \hat{h}(Q)+O_E(1),
\end{equation*}
where $\hat{h}(Q)$ denotes the canonical height of the point $Q$.
On the other hand we know from \cite[Lemma 9]{Silverman} that there is a constant $n_0(E)$ so that for any $\epsilon>0$ and any $n\geq n_0(E)$, 
$$\log{D_n} \geq \left(\frac{1}{3}-\epsilon \right) n^2 \hat{h}(Q)-\log{n}+O_{\epsilon, E}(1).$$
These prove the assertions. 
\end{proof}

We also need the following lemma in the proof of Theorem \ref{elliptic-theorem}. 
\begin{lemma}
\label{A-M}
Suppose that $E$ has complex multiplication by the ring of integers of an imaginary quadratic field $L$. Let $m_{\rm sp}$ be the largest divisor of $m$ composed of primes that split completely in $L$. Then $$\sum_{\substack{{\pi\leq x}\\{m\mid n_\pi(E)}}} 1 \ll_L \frac{\tau(m_{\rm sp})}{m} x.$$
Here $\tau(n)$ denote the divisor function. 
\end{lemma}
\begin{proof}
See \cite[Proposition 2.3]{Akbary-Murty}. 
\end{proof}

We are ready to prove our result in the elliptic curve case.
\begin{proof}[{Proof of Theorem \ref{elliptic-theorem}}]
Since $E$ has $j$-invariant $0$ or $1728$, it has the Weierstrass equation $y^2=x^3+Ax$ or $y^2=x^3+B$. Note that both of these curves have complex multiplication. Here we describe the proof for $y^2=x^3+B$. The proof for $y^2=x^3+Ax$ is analogous.

Since $nQ=(a_n/d_n^2, b_n/d_n^3)$, we have that 
\begin{equation}
 \label{j(E)=0}
b_n^2-a_n^3-B d_n^6=0.
\end{equation}
Applying Conjecture \ref{ABC-conj} to \eqref{j(E)=0} and employing Part (a) of Lemma \ref{d_n and D_n}, we find that
\begin{eqnarray}
\nonumber
\max\{|b_n^2|, |a_n^3|, |B d_n^6|\} &\ll_{\epsilon, E, Q}& \left( \prod_{\pi \mid a_n b_n B \left(\prod_{m\mid n} D_m\right)} \pi \right)^{1+\epsilon} 
\end{eqnarray}
Recall that $D_n=U_n V_n$ is the decomposition of the primitive divisor $D_n$ of $d_n$ to the power-free part $U_n$ and power-full part $V_n$. Thus from the latter inequality we have
\begin{eqnarray}
\label{ABC}
\max\{|b_n^2|, |a_n^3|, |B d_n^6|\} &\ll_{\epsilon, E, Q}&  \left( |a_n b_n B| \prod_{m\mid n} \left( U_m V_m^{1/2}\right)\right)^{1+\epsilon}.
\end{eqnarray}
From \cite[p. 236]{Silverman} we have
\begin{equation*}
|a_nb_n|\leq \sqrt{2} \max \{|a_n^3|, |Bd_n^6|\}^{5/6}.
\end{equation*}
Substituting this bound in \eqref{ABC} and employing Part (b) of Lemma \ref{d_n and D_n} yields
\begin{eqnarray*}
d_n^{1-5\epsilon} \leq \max \{|a_n^3|, |Bd_n^6|\}^{(1-5\epsilon)/6} &\ll_{\epsilon, E, Q}& \left(\prod_{m\mid n} \left( U_m V_m^{1/2}\right)\right)^{1+\epsilon}\\
&\ll_{\epsilon, E, Q}& \left(d_n \prod_{m\mid n} V_m^{-1/2}\right)^{1+\epsilon}.
\end{eqnarray*}
From here we have
\begin{equation}
\label{V bound}
\prod_{m\mid n} V_m \leq C(\epsilon, E, Q) d_n^{\frac{12\epsilon}{1+\epsilon}},
\end{equation}
for some constant $C(\epsilon, E, Q)$ depending on $\epsilon$, $E$, and $Q$.
Now taking the logarithm of two sides of \eqref{V bound} and applying the upper bound for $\log{d_n}$ given in Lemma \ref{size} yields
\begin{equation}
\label{V log bound}
\sum_{m\mid n} \log{V_m} \ll_{\epsilon, E, Q} \frac{12\epsilon}{1+\epsilon} n^2+1.
\end{equation}
On the other hand the lower bound for $\log{D_n}$ given in Lemma \ref{size}  yields
\begin{equation}
 \label{D log bound}
\sum_{m\mid n} \log{D_m} \gg_{E, Q} n^{2}.
\end{equation}
Now, by choosing $\epsilon$ small enough,  \eqref{V log bound} and \eqref{D log bound} yield
\begin{eqnarray}
\label{first bound}
\nonumber
n^{2} &\ll_{\epsilon, E, Q}& \sum_{m\mid n} \log{U_m} \\
\nonumber
&\ll_{\epsilon, E, Q}& \sum_{\substack{{m\mid n} \\ {m<z}}} \log{U_m}+\sum_{\substack{{m\mid n}\\{m>z}}} \log{U_m}\\
&\ll_{\epsilon, E, Q}& \sum_{\substack{{m\mid n} \\ {m<z}}} \log{U_m}+\sum_{\substack{{m\mid n}\\{m>z}}} \sum_{\substack{{\pi\leq P(d_n)}\\{\pi\mid U_m,~\pi\nmid \Delta_E}}}\log{\pi}+\tau(n).
\end{eqnarray}
From Lemma \ref{d_n and D_n} (c) we know that if $\pi\nmid  \Delta_E$ and $\pi\mid U_m$ then ${\rm o}_\pi(Q)=m$. Since ${\rm o}_\pi(Q)\mid n_\pi(E)$ for such $\pi$ we have $m\mid n_\pi(E)$. Employing this fact in the second sum of \eqref{first bound} and applying the upper bound in Lemma \ref{size} for $\log{d_m}$ in the first sum of \eqref{first bound} (note that $\log{U_m} \leq \log{d_m}$) yield 
\begin{equation}
 \label{second bound}
n^{2} \ll_{\epsilon, E, Q} \tau(n) z^2+\left( \log{P(d_n)}\right) \sum_{\substack{{m\mid n }\\{m>z}}} \sum_{\substack{{\pi\leq P(d_n) }\\{m\mid n_\pi(E)}}} 1.
\end{equation}
Since $E$ has complex multiplication we can employ Lemma \ref{A-M} to estimate primes $\pi$ for which $m\mid n_\pi(E)$. Applying Lemma \ref{A-M} in \eqref{second bound} yields
$$n^{2} \ll_{\epsilon, E, Q} \tau(n) z^2+\tau(n) \frac{\left( {P(d_n)}\right)^{1+\epsilon}}{z^{1-\epsilon}}.$$
Now letting $z=n^{1-\epsilon}$ in this inequality implies the result.
\end{proof}


\subsection*{Acknowledgements}
The first author would like to thank Ram Murty for reference \cite{Pr} and the comment regarding the size of $\omega(a^n-1)$, Jeff Bleaney for Example \ref{Jeff}, and Adam Felix for comments on earlier versions of this paper.  



\bibliographystyle{amsalpha}

\end{document}